\documentclass[a4paper]{amsart}
\usepackage{amsmath,amsthm,amssymb,latexsym,epic,bbm,comment,mathrsfs}
\usepackage{graphicx,enumerate,stmaryrd,xcolor}
\usepackage[all,2cell,knot]{xy}
\xyoption{2cell}

\newtheorem{theorem}{Theorem}
\newtheorem{lemma}[theorem]{Lemma}

\newtheorem{proposition}[theorem]{Proposition}
\newtheorem{conjecture}[theorem]{Conjecture}

\usepackage[all]{xy}
\usepackage[active]{srcltx}
\usepackage[parfill]{parskip}
\usepackage{enumerate}

\font\sc=rsfs10

\newcommand{\cP}{\sc\mbox{P}\hspace{1.0pt}}

\font\scc=rsfs7

\newcommand{\ccP}{\scc\mbox{P}\hspace{1.0pt}}

\begin{document}

\title[Kostant's problem for parabolic Verma modules]
{Kostant's problem for parabolic Verma modules}
\author[V.~Mazorchuk and S.~Srivastava]{Volodymyr Mazorchuk and Shraddha Srivastava}

\begin{abstract}
We give a complete combinatorial classification of those parabolic
Verma modules in the principal block of the parabolic category $\mathcal{O}$
associated to a minimal or a maximal parabolic subalgebra of the special linear Lie
algebra for which the answer to Kostant's problem is positive.
\end{abstract}

\maketitle

\section{Introduction and description of the results}\label{s1}

Let $\mathfrak{g}$ be a simple finite dimensional Lie algebra over
the field of complex numbers and $M$ be a $\mathfrak{g}$-module. The 
(associative) algebra of all linear endomorphisms of $M$ has 
various interesting (associative) subalgebras. One of these is 
the image of the universal enveloping algebra $U(\mathfrak{g})$,
which is naturally isomorphic to $U(\mathfrak{g})/\mathrm{Ann}_{U(\mathfrak{g})}(M)$.
The other one is the algebra $\mathcal{L}(M,M)$ of all 
linear endomorphisms of $M$ the adjoint action of $\mathfrak{g}$ 
on which is locally finite. The former algebra is  a
subalgebra of the latter and so it is natural to ask, for which 
$M$, the inclusion
\begin{displaymath}
U(\mathfrak{g})/\mathrm{Ann}_{U(\mathfrak{g})}(M)
\hookrightarrow \mathcal{L}(M,M)
\end{displaymath}
is an isomorphism.

For a given $M$, this is known as {\em Kostant's problem} (for $M$), 
as defined and popularized by Joseph in \cite{Jo}. It is a well-known
and, in general, wide open problem. No complete answer to this problem is
known, even for general simple highest weight modules. However, there 
are various families of modules for which the answer (sometimes positive and
sometimes negative) is known, see Subsection~\ref{s3.6} for a historical overview.

The first non-trivial classical family of modules, for which the answer to 
Kostant's problem was found is the family of Verma modules. Already in \cite{Jo}
it was shown that, for all Verma modules, the answer to the corresponding 
Kostant's problem is positive.

Verma modules are exactly the standard modules with respect to the 
(essentially unique, see \cite{Co}) highest weight structure on the {\em BGG category
$\mathcal{O}$} for $\mathfrak{g}$, see \cite{BGG,Hu}. 
Category $\mathcal{O}$ has a number of different
generalizations. One of these, proposed in \cite{RC}, is called {\em parabolic category
$\mathcal{O}$} and is associated with a choice of a parabolic subalgebra $\mathfrak{p}$
in $\mathfrak{g}$. Just like its ancestor, parabolic category $\mathcal{O}$ 
is a highest weight category. The standard objects with respect to this structure
are {\em parabolic Verma modules}. 

By construction, parabolic Verma modules are quotients of the usual Verma modules. 
In \cite[Section~7.32]{Ja} one can find an argument that shows that 
Kostant's problem has positive answer for any quotient of a Verma module, provided
that the latter is projective in $\mathcal{O}$. In particular, 
Kostant's problem has positive answer for any parabolic Verma module
that is projective in a regular block  of parabolic category
$\mathcal{O}$.  At this stage it is quite 
natural to wonder what the answer to Kostant's problem for 
general parabolic Verma modules will be.

In small ranks and for specific modules, it is sometimes possible to answer 
Kostant's problem by a direct computation. The smallest non-trivial example
of parabolic category $\mathcal{O}$ is for $\mathfrak{g}=\mathfrak{sl}_3$
where we take the parabolic subalgebra corresponding to one of the two simple 
roots. The regular block of the corresponding parabolic category $\mathcal{O}$
contains three parabolic Verma modules. As we already mentioned above, for the
projective parabolic Verma module, the answer to Kostant's problem is positive.
For the simple parabolic Verma module, the answer to Kostant's problem 
is also known to be positive, see \cite{GJ}. It was fairly surprising for us to 
find out, by a direct computation, that for the (remaining) third 
parabolic Verma module the  answer to Kostant's problem is negative.
This was clear evidence that determining the  answer to Kostant's problem
for parabolic Verma modules is a non-trivial problem as first one has to 
guess for which parabolic Verma modules the answer should be positive and 
for which it should be negative. So far, there is only one general result
in the literature which answers  Kostant's problem for a 
fairly natural general family of modules where both the cases of the positive
and the negative answers occur. This is the family of simple highest weight
modules over the special linear Lie algebra indexed by fully commutative
permutations, see the recent preprint \cite{MMM}. The main result of the
present paper is the second one of this kind.

In the present paper we study Kostant's problem for parabolic Verma modules
over the special linear Lie algebra in two ``opposite''
cases. The first case is the case of a minimal parabolic
subalgebra, that is the case when the semi-simple part of the parabolic
subalgebra is given by one simple root. In this case, parabolic Verma
modules are quotients of Verma modules by Verma submodules. These
kinds of quotients were studied in \cite{KMM}, where the emphasis was made on
a description of the socle for such a quotient. This description is crucial
for our proofs. Two other important ingredients in our arguments are:

\begin{itemize}
\item An adaptation of K{\aa}hrstr{\"o}m's conjectural combinatorial 
reformulation for Kos\-tant's problem, see \cite[Conjecture~1.2]{KMM2}.
\item An appropriate analogue of the $2$-representation theoretic 
reformulation of Kostant's problem, see \cite[Subsection~8.3]{KMM2}.
\end{itemize}

The second case we consider is that of a maximal parabolic subalgebra,
that is the case when the semi-simple part of the parabolic subalgebra
misses exactly one simple root.
In this case there is a very explicit diagrammatic description of 
the corresponding parabolic category $\mathcal{O}$ worked out in
\cite{BS11a,BS11b}. To answer Kostant's problem for 
parabolic Verma modules in this case, we combine this description with the 
main result of \cite{MMM}. Surprisingly, in this case the positive
answer is really rare. In fact, there are several infinite families 
of examples, where the answer is positive only for two parabolic Verma modules,
namely for the projective one and for the simple one, for which the 
answer is always positive (just like 
in the $\mathfrak{sl}_3$ example mentioned above).

The paper is organized as follows: In Section~\ref{s2} we collect all
relevant preliminaries on category~$\mathcal{O}$. In Section~\ref{s3}
we give a general perspective on Kostant's problem and its history.
In Section~\ref{s5} we study the case of a minimal parabolic subalgebra.
Our main result here is Theorem~\ref{mainthm} which completely
answers Kostant's problem for parabolic Verma modules in this case
and also relates it to K{\aa}hrstr{\"o}ms combinatorial
reformulation.
This theorem can be found in 
Subsection~\ref{s5.2}. In Section~\ref{s9} we study the case of a
maximal parabolic subalgebra. Our main result here is Theorem~\ref{mainthm2}
which completely answers Kostant's problem for 
parabolic Verma modules in this case.
This theorem can be found in Subsection~\ref{s9.5}. We finish the paper with
some general observations and speculations in Section~\ref{s8}.

\vspace{1mm}

\subsection*{Acknowledgments}
The first author is partially supported by the Swedish Research Council.

\section{Category $\mathcal{O}$ preliminaries}\label{s2}

\subsection{Setup}\label{s2.1}

In this paper, we work over the field $\mathbb{C}$ of complex numbers. 

As already mentioned above, $\mathfrak{g}$ is a simple finite dimensional 
complex Lie algebra. We fix a triangular decomposition 
\begin{equation}\label{eq1.1}
\mathfrak{g}=\mathfrak{n}_-\oplus \mathfrak{h}\oplus \mathfrak{n}_+, 
\end{equation}
where $\mathfrak{h}$ is some fixed Cartan subalgebra. Denote by $W$ 
the Weyl group of $\mathfrak{g}$ and by $S$ the set of simple reflections 
in $W$ which corresponds to the triangular decomposition \eqref{eq1.1} above.
As usual, we denote by $w_0$ the longest element of $W$.
Let $\leq$ denote the usual {\em Bruhat order} on $W$
and $\ell$ the length function on $W$.

In this paper we will mostly consider $\mathfrak{g}=\mathfrak{sl}_n$ with the
standard triangular decomposition given by the upper triangular, the diagonal and
the lower triangular matrices. In this case we have $W\cong \mathbf{S}_n$, the symmetric
group on $\{1,2,\dots,n\}$, with $S$ being the set of elementary transpositions.

\subsection{Category $\mathcal{O}$}\label{s2.2}

Associated to the triangular decomposition \eqref{eq1.1}, we have 
the BGG category $\mathcal{O}$ defined as the full subcategory of the
category of finitely generated $\mathfrak{g}$-modules consisting of all
modules on which the action of $\mathfrak{h}$ is diagonalizable and the
action of $\mathfrak{n}_+$ is locally finite. We refer to \cite{BGG,Hu} 
for details.

\subsection{Principal block}\label{s2.3}

Consider the principal block $\mathcal{O}_0$ of $\mathcal{O}$.
This block is defined as that indecomposable direct summand of 
$\mathcal{O}$ which contains the trivial $\mathfrak{g}$-module.
The isomorphism classes of simple objects in $\mathcal{O}_0$
are naturally indexed by the elements of $W$. For $w\in W$,
we have the corresponding simple highest weight module
$L_w:=L(w\cdot 0)$. Here $0\in \mathfrak{h}^*$ denotes the zero weight
and $\cdot$ denotes the dot-action of $W$ on $\mathfrak{h}^*$.
For $\lambda\in\mathfrak{h}^*$, the dot-action is defined
as follows: $w\cdot \lambda=w(\lambda+\rho)-\rho$. Here
$\rho$ is the half of the sum of positive roots.

Furthermore, for $w\in W$, we have the Verma cover $\Delta_w$ of 
$L_w$ and the indecomposable projective cover $P_w$ of $L_w$. 
We denote by $A$ the opposite of the endomorphism algebra of a 
multiplicity-free projective generator of $\mathcal{O}_0$. 
The algebra $A$ is a finite dimensional and
associative algebra. Moreover, the category $A$-mod of all
finite dimensional $A$-modules is equivalent to the category 
$\mathcal{O}_0$.

\subsection{Parabolic category $\mathcal{O}$}\label{s2.35}

Let $\mathfrak{p}$ be a parabolic subalgebra of $\mathfrak{g}$ containing 
$\mathfrak{h}\oplus \mathfrak{n}_+$.  The corresponding parabolic 
category $\mathcal{O}^\mathfrak{p}$ is defined as the full subcategory
of $\mathcal{O}$ that consists of all objects the action of 
$\mathfrak{p}$ on which is locally finite, see \cite{RC}. We further have the 
principal block $\mathcal{O}_0^\mathfrak{p}$ of $\mathcal{O}^\mathfrak{p}$
defined as $\mathcal{O}_0\cap \mathcal{O}^\mathfrak{p}$.

Let $W_\mathfrak{p}$ be the parabolic subgroup of $W$ corresponding
to $\mathfrak{p}$. Denote by $W^{\mathfrak{p}}_{\mathrm{short}}$ 
the set of the shortest
coset representatives in ${W_\mathfrak{p}}\hspace{-1mm}\setminus W$.
The category $\mathcal{O}_0^\mathfrak{p}$ is the Serre subcategory
of $\mathcal{O}_0$ generated by the simple objects $L_w$, where
$w\in W^{\mathfrak{p}}_{\mathrm{short}}$. The category 
$\mathcal{O}_0^\mathfrak{p}$ is equivalent to $A_{\mathfrak{p}}$-mod
for a certain quotient $A_{\mathfrak{p}}$ of $A$.

For $w\in W^{\mathfrak{p}}_{\mathrm{short}}$, we denote by 
$\Delta^\mathfrak{p}_w$ the corresponding parabolic Verma module
and by $P_w^\mathfrak{p}$ the corresponding indecomposable 
projective module in $\mathcal{O}_0^\mathfrak{p}$.

\subsection{Projective functors}\label{s2.4}

Following \cite{BG}, for every $w\in W$,  there is a unique, up to 
isomorphism, indecomposable projective endofunctor $\theta_w$ 
of the category $\mathcal{O}_0$ such that $\theta_w\, P_e\cong P_w$.
Any indecomposable projective endofunctor of $\mathcal{O}_0$
is isomorphic to $\theta_w$, for some $w\in W$.

Consider the monoidal category $\cP$ of all projective endofunctors
of $\mathcal{O}_0$. By Soergel's combinatorial description,
see \cite{So3}, the monoidal category $\cP$ is monoidally equivalent 
to the monoidal category of Soergel bimodules over the coinvariant 
algebra of the Weyl group $W$.

\subsection{Twisting functors}\label{s2.45}

For $w\in W$, we denote by $\top_w$ the corresponding
twisting functor, see \cite{AS,KhM}. For $x,y\in W$ such that
$\ell(xy)=\ell(x)+\ell(y)$, we have 
$\top_x\Delta_y\cong \Delta_{xy}$. The main, for us, properties of 
twisting functors are that they functorially commute with 
projective functors, that they are acyclic on Verma modules
and that the corresponding left derived functors are 
equivalences, see \cite{AS}.

\subsection{Graded lift}\label{s2.5}

By \cite{So}, the algebra $A$ is Koszul and hence it is equipped
with the corresponding Koszul $\mathbb{Z}$-grading.
We therefore have the category  ${}^\mathbb{Z}\mathcal{O}_0$ of
finite dimensional $\mathbb{Z}$-graded $A$-modules.
We denote by $\langle{}_-\rangle$ the functor which shifts the grading
with the convention that $\langle 1\rangle$ shifts degree $0$ to degree
$-1$.

All structural modules in the category $\mathcal{O}_0$ admit graded 
lifts, see \cite{St} for details. These graded lifts are unique, 
up to isomorphism and grading shift, for all indecomposable modules.

We fix standard graded lifts of all indecomposable structural modules
and will use for these standard graded lifts the same notation
as for the corresponding ungraded objects in $\mathcal{O}_0$,
abusing notation. Similarly, we have the standard graded lifts of 
the indecomposable projective functors, see \cite{St}.
We denote by $\cP^\mathbb{Z}$ the corresponding monoidal category of 
graded projective endofunctors of ${}^\mathbb{Z}\mathcal{O}_0$.

The algebra $A_{\mathfrak{p}}$ is a graded quotient of 
$A$ and we have the corresponding graded analogue 
${}^\mathbb{Z}\mathcal{O}_0^\mathfrak{p}$ of the category 
$\mathcal{O}_0^\mathfrak{p}$.

\subsection{Hecke algebra combinatorics}\label{s2.6}

Consider the Hecke algebra $\mathbf{H}=\mathbf{H}(W,S)$ of $W$
in the normalization of \cite{So2}. It is an algebra over the
commutative ring $\mathbb{Z}[v,v^{-1}]$. This algebra has the 
standard basis $\{H_w\,:\,w\in W\}$  as well as the 
Kazhdan--Lusztig basis $\{\underline{H}_w\,:\,w\in W\}$, 
defined in  \cite{KL}.

The Grothendieck group $\mathbf{Gr}[{}^\mathbb{Z}\mathcal{O}_0]$ 
of the category ${}^\mathbb{Z}\mathcal{O}_0$ is isomorphic to the algebra
$\mathbf{H}$ by sending $[\Delta_w]$ to $H_w$. This isomorphism
intertwines $\langle 1\rangle$ and multiplication with $v^{-1}$ and
sends $[P_w]$ to $\underline{H}_w$, see \cite{BB81,BK81}.

The split Grothendieck group $\mathbf{Gr}_\oplus[\cP^\mathbb{Z}]$
of the monoidal category $\cP^\mathbb{Z}$
is isomorphic to $\mathbf{H}$ by sending $[\theta_w]$ to $\underline{H}_w$.
The defining action of $\cP^\mathbb{Z}$ on the category $\mathcal{O}_0^\mathbb{Z}$
decategorifies to the right regular representation of $\mathbf{H}$, see \cite{So3}.

Denote by $\leq_L$, $\leq_R$ and $\leq_J$ the 
left, right and two-sided Kazhdan-Lusztig preorders on $W$, 
respectively.  The equivalence classes associated to these 
preorders are called left, right and two-sided cells. Each left
and each right cell contains a distinguished involution, which is 
called the Duflo involution. For symmetric groups, 
all involutions are, in fact,  Duflo involutions.

The singletons $\{e\}$ and $\{w_0\}$ are two-sided cells.
Further, we have the {\em small} two-sided cell $\mathcal{J}_s$
containing all simple reflections. In the case of the symmetric
group $\mathbf{S}_n$, the small cell contains exactly $n-1$ left and
$n-1$ right cells and consists of all elements having exactly
one unique expression. Multiplying the elements of this small
two-sided cell with $w_0$, we obtain the {\em penultimate}
two-sided cell $\mathcal{J}_p$.

\subsection{Bigrassmannian permutations and socles of cokernels of inclusions of Verma modules}\label{s2.7}

Recall that an element $w\in W$ is called {\em bigrassmannian} provided that 
there is a unique simple reflection $s$ such that $sw<w$ and 
there is a unique simple reflection $t$ (note that $t$ might be equal to $s$) 
such that $wt<w$. By \cite[Theorem~1]{KMM}, in the case of the algebra
$\mathfrak{sl}_n$, for $w\in \mathbf{S}_n$, the module 
$\Delta_e/(\Delta_w\langle -\ell(w)\rangle)$ has simple
socle if and only if $w$ is bigrassmannian.

Furthermore, the map $w\mapsto \mathrm{Soc}(\Delta_e/(\Delta_w\langle -\ell(w)\rangle))$
is a bijection between the set of all bigrassmannian elements in $\mathbf{S}_n$
and graded subquotients of $\Delta_e$ of the form $L_w\langle i\rangle$,
where $w\in \mathcal{J}_p$. If $w\in \mathbf{S}_n$ is not bigrassmannian, then the
socle of $\Delta_e/(\Delta_w\langle -\ell(w)\rangle$ is given by the 
elements that correspond, under the above bijection, to those bigrassmannian
elements in the Bruhat interval $[e,w]$ that are maximal among all 
bigrassmannian elements in this interval with respect to the Bruhat order.

\section{Kostant's problem}\label{s3}

\subsection{Harish-Chandra bimodules}\label{s3.2}

Recall from \cite[Kapitel~6]{Ja}, that a  $\mathfrak{g}$-$\mathfrak{g}$-bimodule $M$
is called a {\em Harish-Chandra bimodule} provided that it is finitely generated 
as a bimodule and, additionally,  provided that the adjoint action of 
$\mathfrak{g}$ on $M$ is locally finite with finite 
 multiplicities for composition factors.
For example, for any $M\in \mathcal{O}_0$, the bimodule
$U(\mathfrak{g})/\mathrm{Ann}_{U(\mathfrak{g})}(M)$ is a Harish-Chandra bimodule.

Another example can be given as follows: for two objects $M$ and $N$
in $\mathcal{O}_0$, the space $\mathcal{L}(M,N)$ of all linear maps from $M$ to 
$N$ on which the adjoint action of $\mathfrak{g}$ is locally finite forms 
a Harish-Chandra bimodule. In the case $M=N$, we have a natural inclusion 
$U(\mathfrak{g})/\mathrm{Ann}_{U(\mathfrak{g})}(M)\subset  \mathcal{L}(M,M)$.

\subsection{Classical Kostant's problem}\label{s3.3}

As we have already mentioned in the introduction, for a
$\mathfrak{g}$-module $M$, the corresponding Kostant's problem, 
as formulated in \cite{Jo}, is the following:

{\bf Kostant's  Problem.} Is the embedding 
$$U(\mathfrak{g})/\mathrm{Ann}_{U(\mathfrak{g})}(M)\hookrightarrow 
\mathcal{L}(M,M)$$ an isomorphism?

We will denote by $\mathbf{K}(M)
\in\{\mathtt{true},\mathtt{false}\}$ the logical value of the claim
``the embedding  $U(\mathfrak{g})/\mathrm{Ann}_{U(\mathfrak{g})}(M)\hookrightarrow 
\mathcal{L}(M,M)$ is an isomorphism''. 

\subsection{K{\aa}hrstr{\"o}m's conjecture}\label{s3.4}

The following conjecture, formulated in \cite[Conjecture~1.2]{KMM},
is due to Johan K{\aa}hrstr{\"o}m:

\begin{conjecture}\label{conj2}
Let $d$ be an involution in the symmetric group $\mathbf{S}_n$. 
Then the following assertions are equivalent:

\begin{enumerate}
\item\label{conj2.1} $\mathbf{K}(L_d)=\mathtt{true}$.
\item\label{conj2.2} For $x,y\in W$ such that $x\neq y$, $\theta_x L_d\neq 0$ and
$\theta_y L_d\neq 0$, we have $\theta_x L_d\not\cong \theta_y L_d$.
\item\label{conj2.3} For $x,y\in W$ such that $x\neq y$, $\theta_x L_d\neq 0$ and
$\theta_y L_d\neq 0$, we have $[\theta_x L_d]\neq [\theta_y L_d]$
in $\mathbf{Gr}[\mathcal{O}_0^\mathbb{Z}]$.
\item\label{conj2.4} For $x,y\in W$ such that $x\neq y$, $\theta_x L_d\neq 0$ and
$\theta_y L_d\neq 0$, we have $[\theta_x L_d]\neq [\theta_y L_d]$
in $\mathbf{Gr}[\mathcal{O}_0]$.
\end{enumerate} 
\end{conjecture}

\subsection{Kostant's problem and $2$-representation theory}\label{s3.5}

Given $M\in\mathcal{O}$, we can consider its annihilator $\mathrm{Ann}_{\ccP}(M)$,
which is a left monoidal ideal of $\cP$. The quotient $\cP/\mathrm{Ann}_{\ccP}(M)$
is, naturally, a birepresentation of $\cP$.

Alternatively, the additive closure $\mathrm{add}(\cP\, M)$ is also a
birepresentation of $\cP$. Mapping 
\begin{displaymath}
\cP\ni\theta\mapsto \theta(M)\in  \mathrm{add}(\cP\, M)
\end{displaymath}
defines an injective  morphism of birepresentation from
$\cP/\mathrm{Ann}_{\ccP}(M)$ to $\mathrm{add}(\cP\, M)$.
The arguments from \cite[Subsection~8.3]{KMM2} show that 
this morphism is an equivalence if and only if 
$\mathbf{K}(M)=\mathtt{true}$. 

In fact, that 
the birepresentation $\mathrm{add}(\cP\, M)$
of $\cP$ is given by the algebra $\mathcal{L}(M,M)$
follows from \cite[Subsection~6.8]{Ja}.
The principal birepresentation of $\cP$, that is the natural
left action of $\cP$ on $\cP$ corresponds to the 
identity $1$-morphism of $\cP$ by \cite[Formula~(4.2)]{MMMT}.
In the world of Harish-Chandra bimodules, the 
algebra $U(\mathfrak{g})$ surjects onto the
identity $1$-morphism of $\cP$. 
Consequently, the induced action of $\cP$ on
$\cP/\mathrm{Ann}_{\ccP}(M)$ corresponds to the algebra
given by the quotient of the identity $1$-morphism of $\cP$ 
(and hence also of $U(\mathfrak{g})$) by the
annihilator of $M$. The algebra homomorphism 
$U(\mathfrak{g})\to \mathcal{L}(M,M)$ is a manifestation of
the fact that $\cP/\mathrm{Ann}_{\ccP}(M)$ is a subbirepresentation of
$\mathrm{add}(\cP\, M)$ and, therefore, the positive answer to 
Kostant's problem for $M$ reformulates into the property of 
this natural embedding being an equivalence.

This reformulation allows us to connect the answers to 
Kostant's problem for modules related by twisting functors.

\begin{lemma}\label{lem3.5-1}
Let $M\in\mathcal{O}$ and $w\in W$ be such that 
$\mathcal{L}\top_w(M)\cong \top_w(M)\in \mathcal{O}$. Then 
$\mathbf{K}(M)=\mathbf{K}(\top_w(M))$.
\end{lemma}

\begin{proof}
This follows from the birepresentation-theoretical reformulation
of Kostant's problem described above combined with the properties of twisting
functors mentioned in Subsection~\ref{s2.45}.
\end{proof}

\subsection{Known results on Kostant's problem}\label{s3.6}

Below we present a list of the results on Kostant's problem 
which we could find in existing literature.

\begin{itemize}
\item In the case of symmetric groups, the value $\mathbf{K}(L_w)$ is 
constant on Kazhdan--Lusztig left cells, as shown in \cite[Theorem~61]{MS}.
\item Let $w_0^\mathfrak{p}$ denote the longest element 
in $W_{\mathfrak{p}}$. Then $\mathbf{K}(L_{w_0^\mathfrak{p}w_0})=
\mathtt{true}$, see \cite[Theorem~4.4]{GJ} and \cite[Section~7.32]{Ja}.
\item If $s\in W^{\mathfrak{p}}$ is a simple reflection,
then $\mathbf{K}(L_{sw_0^\mathfrak{p}w_0})=
\mathtt{true}$, see \cite[Theorem~1]{Ma}.
\item \cite[Theorem~1.1]{Ka} relates the answer to Kostant's problem 
for certain highest weight $\mathfrak{g}$-modules to the 
answer to Kostant's problem for certain highest weight modules
over the Levi quotients of  $\mathfrak{p}$.
\item \cite[Theorem~1]{KaM} proposes a module-theoretic 
characterization of the equality $\mathbf{K}(L_w)=\mathtt{true}$.
\item In \cite[Section~4]{KaM}, one can find a complete answer to Kostant's 
problem for simple highest weight modules over $\mathfrak{sl}_n$, where 
$n=2,3,4,5$; and a partial answer for the same problem for $\mathfrak{sl}_6$. 
Some further $\mathfrak{sl}_6$-answers were computed in \cite[Section~6]{Ka}. 
The highest weight $\mathfrak{sl}_6$-story was completed in
\cite[Subsection~10.1]{KMM2}.
\item The paper \cite{MMM} provides a complete answer to 
Kostant's problem for simple highest weight $\mathfrak{sl}_n$-modules
indexed by fully commutative permutations.
\end{itemize}

\section{Main results: minimal parabolic subalgebras}\label{s5}

\subsection{Preliminaries}\label{s5.1}

For $n\geq 1$, we consider the symmetric group $\mathbf{S}_n$. We view elements
of $\mathbf{S}_n$ as functions on $\underline{n}:=\{1,2,\dots,n\}$ which we 
compose from right to left. 

Fix $k\in\{1,2,\dots,n-1\}$. For $1\leq i< j\leq n$, denote by
$\tau^{n,k}_{i,j}$ the unique element of $\mathbf{S}_n$ such that 
\begin{itemize}
\item $\tau^{n,k}_{i,j}(i)=k$;
\item $\tau^{n,k}_{i,j}(j)=k+1$;
\item for all $s<t$ in $\underline{n}\setminus\{i,j\}$ we have
$\tau^{n,k}_{i,j}(s)<\tau^{n,k}_{i,j}(t)$.
\end{itemize}
We denote by $\mathcal{X}_{n,k}$ the set of all these 
$\tau^{n,k}_{i,j}$. Clearly, $|\mathcal{X}_{n,k}|=\binom{n}{2}$.
We further split $\mathcal{X}_{n,k}$ into two disjoint subsets:
$\mathcal{X}_{n,k}=\mathcal{X}_{n,k}^+\coprod\mathcal{X}_{n,k}^-$,
where
\begin{displaymath}
\mathcal{X}_{n,k}^+=\{\tau^{n,k}_{i,i+1}\,:\,i=1,2,\dots,n-1\}
\quad\text{ and }\quad
\mathcal{X}_{n,k}^-=\mathcal{X}_{n,k}\setminus \mathcal{X}_{n,k}^+.
\end{displaymath}
Here is an example for $n=4$ and $k=2$, where in the first row
we list all elements in $\mathcal{X}_{4,3}^+$ and in the
second row we list all elements in $\mathcal{X}_{4,3}^-$
(for convenience, we color in {\color{magenta}magenta}
the important edges in these graphs
which go to $k=2$ and $k+1=3$):
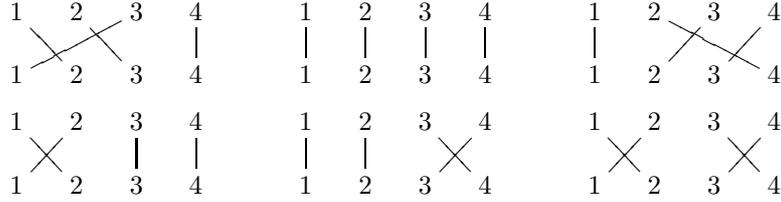
\begin{displaymath}
\xymatrix@C=4mm@R=4mm{1\ar@{-}@[magenta][dr]&2\ar@{-}@[magenta][dr]
&3\ar@{-}[dll]&4\ar@{-}[d]\\1&2&3&4}\quad\quad\quad
\xymatrix@C=4mm@R=4mm{1\ar@{-}[d]&2\ar@{-}@[magenta][d]&
3\ar@{-}@[magenta][d]&4\ar@{-}[d]\\1&2&3&4}\quad\quad\quad
\xymatrix@C=4mm@R=4mm{1\ar@{-}[d]&2\ar@{-}[drr]&
3\ar@{-}@[magenta][dl]&4\ar@{-}@[magenta][dl]\\1&2&3&4}
\end{displaymath}
\begin{displaymath}
\xymatrix@C=4mm@R=4mm{1\ar@{-}@[magenta][dr]&2\ar@{-}[dl]&
3\ar@{-}@[magenta][d]&4\ar@{-}[d]\\1&2&3&4}\quad\quad\quad
\xymatrix@C=4mm@R=4mm{1\ar@{-}[d]&2\ar@{-}@[magenta][d]&
3\ar@{-}[dr]&4\ar@{-}@[magenta][dl]\\1&2&3&4}\quad\quad\quad
\xymatrix@C=4mm@R=4mm{1\ar@{-}@[magenta][dr]&2\ar@{-}[dl]&
3\ar@{-}[dr]&4\ar@{-}@[magenta][dl]\\1&2&3&4}
\end{displaymath}

We denote by $G_k$ the centralizer of $(k,k+1)$ in $\mathbf{S}_n$.
The group $G_k$ consists of all elements of $\mathbf{S}_n$ which leave
$\{k,k+1\}$ invariant, in particular, $G_k$ is the direct product
of the symmetric groups on $\{k,k+1\}$ and $\underline{n}\setminus\{k,k+1\}$.
Hence $|G_k|=2\cdot(n-2)!$. Let $\hat{G}_k$ denote the subgroup of
$G_k$ consisting of all elements which fix both $k$ and $k+1$.
Then $\hat{G}_k$ is naturally isomorphic to the symmetric group 
on $\underline{n}\setminus\{k,k+1\}$ and $|\hat{G}_k|=(n-2)!$.

\subsection{Formulation}\label{s5.2}

For $k\in\{1,2,\dots,n-1\}$, consider the parabolic subalgebra
$\mathfrak{p}=\mathfrak{p}_k$ of $\mathfrak{sl}_n$ corresponding to the simple
reflection $(k,k+1)$.
The composition $\circ$ in $\mathbf{S}_n$ gives rise to a bijection
\begin{displaymath}
\hat{G}_k\times \mathcal{X}_{n,k} \to  (\mathbf{S}_n)^{\mathfrak{p}}_{\mathrm{short}}.
\end{displaymath}

\begin{theorem}\label{mainthm}
For $w\in W^{\mathfrak{p}}_{\mathrm{short}}$, the following assertions are equivalent:
\begin{enumerate}[$($a$)$]
\item\label{mainthm.1} $\mathbf{K}(\Delta^{\mathfrak{p}}_w)=\mathtt{true}$.
\item\label{mainthm.2} For $x,y\in W$ such that $x\neq y$, 
$\theta_x \Delta^{\mathfrak{p}}_w\neq 0$ and
$\theta_y \Delta^{\mathfrak{p}}_w\neq 0$, we have $\theta_x \Delta^{\mathfrak{p}}_w\not\cong \theta_y \Delta^{\mathfrak{p}}_w$ (as ungraded modules).
\item\label{mainthm.3} For all $x,y\in W$ such that $x\neq y$, 
$\theta_x \Delta^{\mathfrak{p}}_w\neq 0$ and
$\theta_y \Delta^{\mathfrak{p}}_w\neq 0$, we have 
$[\theta_x \Delta^{\mathfrak{p}}_w]\neq [\theta_y \Delta^{\mathfrak{p}}_w]\langle i\rangle$,
for $i\in\mathbb{Z}$,
in $\mathbf{Gr}[\mathcal{O}_0^\mathbb{Z}]$.
\item\label{mainthm.4} For all $x,y\in W$ such that $x\neq y$, 
$\theta_x \Delta^{\mathfrak{p}}_w\neq 0$ and
$\theta_y \Delta^{\mathfrak{p}}_w\neq 0$, we have 
$[\theta_x \Delta^{\mathfrak{p}}_w]\neq [\theta_y \Delta^{\mathfrak{p}}_w]$
in $\mathbf{Gr}[\mathcal{O}_0]$.
\item\label{mainthm.5} $w\in \hat{G}_k\circ\mathcal{X}_{n,k}^+$.
\item\label{mainthm.6} The annihilator of $\Delta^{\mathfrak{p}}_w$ in $U(\mathfrak{g})$
is a primitive ideal.
\end{enumerate}
\end{theorem}

We would like to point out the subtle difference between
Theorem~\ref{mainthm}\eqref{mainthm.3} and
Conjecture~\ref{conj2}\eqref{conj2.3}.
Theorem~\ref{mainthm} provides a complete answer to 
Kostant's problem for parabolic Verma modules in the above setup.

\subsection{Positive cases}\label{s5.3}
Let $w\in \hat{G}_k\circ\mathcal{X}_{n,k}^+$. For this $w$ we will prove
that all assertions in Theorem~\ref{mainthm} hold. Note that
Theorem~\ref{mainthm}\eqref{mainthm.5} is obvious.
Also note the obvious implications
Theorem~\ref{mainthm}\eqref{mainthm.4}
$\Rightarrow$Theorem~\ref{mainthm}\eqref{mainthm.3}
$\Rightarrow$Theorem~\ref{mainthm}\eqref{mainthm.2}.

Assume that $w=\sigma\circ \tau^{n,k}_{i,i+1}$, for some
$\sigma\in \hat{G}_k$ and some $i\in\{1,2,\dots,n-1\}$.
We have the short exact sequence
\begin{displaymath}
0\to\Delta_{(i,i+1)}\to\Delta_e\to\Delta^{\mathfrak{p}_i}_e\to 0. 
\end{displaymath}
Since $\Delta^{\mathfrak{p}_i}_e$ is a quotient of $\Delta_e$,
we have $\mathbf{K}(\Delta^{\mathfrak{p}_i}_e)=\mathtt{true}$
by \cite[Section~7.32]{Ja}.

Note that $\tau^{n,k}_{i,i+1}\circ (i,i+1)=(k,k+1)\circ\tau^{n,k}_{i,i+1}$
and that $\sigma\circ (i,i+1)=(i,i+1)\circ \sigma$ since $\sigma\in \hat{G}_k$.
From the definition of $\hat{G}_k$ it also follows that 
$\ell((k,k+1)w)=\ell(w)+1$. Therefore, applying $\top_{w}$ to the above
short exact sequence, we obtain
\begin{displaymath}
0\to\Delta_{(k,k+1)w}\to\Delta_w\to\Delta^{\mathfrak{p}}_w\to 0. 
\end{displaymath}
Since $\top_{w}$ is exact and acyclic on Verma modules,
we can now apply Lemma~\ref{lem3.5-1} and conclude that 
$\mathbf{K}(\Delta^{\mathfrak{p}}_w)=\mathtt{true}$, giving
Theorem~\ref{mainthm}\eqref{mainthm.1}.

The non-zero objects of the form $\theta_x\, \Delta^{\mathfrak{p}_i}_e$
are exactly the indecomposable projective objects in 
$\mathcal{O}_0^{\mathfrak{p}_i}$. Since the latter category is a 
highest weight category, it has finite global dimension and thus
the images of the indecomposable projective objects form a basis in
the Grothendieck group of this category. In particular, these images
are linearly independent and therefore different.

As $\mathcal{L}\top_w$ is a derived equivalence, the classes of 
the non-zero modules of the form $\top_w \theta_x\, \Delta^{\mathfrak{p}_i}_e$
are also linearly independent in the corresponding Grothendieck group.
In particular, they are different, which implies 
Theorem~\ref{mainthm}\eqref{mainthm.4} (and thus also both 
Theorem~\ref{mainthm}\eqref{mainthm.3} and 
Theorem~\ref{mainthm}\eqref{mainthm.2}, see above).

It remains to prove Theorem~\ref{mainthm}\eqref{mainthm.6}. 
The construction of twisting functors via localization
(see \cite{KhM}) immediately implies that
$\mathrm{Ann}_{U(\mathfrak{g})}(\top_w M)\supset \mathrm{Ann}_{U(\mathfrak{g})}(M)$.
The right adjoints of twisting functors, which can be
realized via Enright completion functors when acting on Verma modules,
have the same property. This means that 
$\mathrm{Ann}_{U(\mathfrak{g})}(\Delta^{\mathfrak{p}}_w)=
\mathrm{Ann}_{U(\mathfrak{g})}(\Delta^{\mathfrak{p}_i}_e)$.
The fact that the latter annihilator coincides with the annihilator of the
(simple) socle of $\Delta^{\mathfrak{p}_i}_e$ is a standard fact,
for example, it follows from \cite[Proposition~5.1]{Ka}.
Therefore $\mathrm{Ann}_{U(\mathfrak{g})}(\Delta^{\mathfrak{p}_i}_e)$ 
(and thus also $\mathrm{Ann}_{U(\mathfrak{g})}(\Delta^{\mathfrak{p}}_w)$)
is a primitive ideal.

This completes the proof of all positive cases.

\subsection{Negative cases}\label{s5.4}
Let $w\in \hat{G}_k\circ\mathcal{X}_{n,k}^-$. For this $w$ we will prove
that all assertions in Theorem~\ref{mainthm} fail. Note that
$\neg$Theorem~\ref{mainthm}\eqref{mainthm.5} is obvious.
Also note the obvious implication
$\neg$Theorem~\ref{mainthm}\eqref{mainthm.2}
$\Rightarrow$ $\neg$Theorem~\ref{mainthm}\eqref{mainthm.4}.

The argument with twisting functors used in the previous subsection 
reduces the present consideration to the case $w\in \mathcal{X}_{n,k}^-$,
that is $w=\tau^{n,k}_{i,j}$, for some $j\geq i+2$. We note that this
$w$ is not bigrassmannian. Let us use the results of \cite{KMM}
to analyze the subquotients of
$\Delta^\mathfrak{p}_{\tau^{n,k}_{i,j}}=
\Delta_{\tau^{n,k}_{i,j}}/\Delta_{(k,k+1)\tau^{n,k}_{i,j}}$
of the form $L_x$, where  $x\in \mathcal{J}_p$. We will loosely call
such subquotients {\em penultimate}. 

To understand the penultimate subquotients in 
$\Delta^\mathfrak{p}_{\tau^{n,k}_{i,j}}$, we just need
to compare such subquotients for the modules $\Delta_e/\Delta_{\tau^{n,k}_{i,j}}$
and for $\Delta_e/\Delta_{(k,k+1)\tau^{n,k}_{i,j}}$. By
\cite[Theorem~1]{KMM}, the difference between these two sets 
is in bijection with the bigrassmannian elements in the Bruhat complement
$[e,(k,k+1)\tau^{n,k}_{i,j}]\setminus [e,\tau^{n,k}_{i,j}]$. Taking into account
the very explicit forms of $\tau^{n,k}_{i,j}$ and $(k,k+1)\tau^{n,k}_{i,j}$,
we can determine all bigrassmannian elements in this Bruhat complement. 
This, however, will require  consideration of a number of different cases.

To simplify our notation, we denote $\tau^{n,k}_{i,j}$ simply by $\tau$.
We also denote the elementary transposition $(l,l+1)$ by $s_l$.

\subsection{Case~1: $k=i$}

In this case $\tau=s_{k+1}s_{k+2}\dots s_{j-1}$ and the Bruhat complement
$[e,s_k\tau]\setminus[e,\tau]$ contains the following bigrassmannian elements:
$s_{k}$, $s_{k}s_{k+1}$,\dots, $s_k\tau$. Note that there are at least two
such elements and that they all have different right descents.

Consequently, the module $\Delta^\mathfrak{p}_{\tau^{n,k}_{i,j}}$
contains at least two (pairwise) non-isomorphic penultimate subquotients.
From \cite[Proposition~22]{KMM} it follows that they appear in different degrees.
Let $L_x$ and $L_y$ be two of them (living in some degrees). Note that
both $x$ and $y$ belong to the same Kazhdan-Lusztig right cell,
namely the unique right cell inside $\mathcal{J}_p$ which 
indexes some simples in $\mathcal{O}_0^{\mathfrak{p}}$.
Since $x\neq y$, these two elements belong to different Kazhdan-Lusztig left cells.
In particular, from \cite[Lemma~12]{MM1} it follows that $\theta_{x^{-1}}$ kills
all simples in $\mathcal{O}_0^{\mathfrak{p}}$ except for 
$L_x$ while $\theta_{y^{-1}}$ kills
all simples in $\mathcal{O}_0^{\mathfrak{p}}$ except for 
$L_y$.

Therefore, up to graded shift, $\theta_{x^{-1}}\Delta^{\mathfrak{p}}_w$
is isomorphic to $\theta_{x^{-1}}L_x$. Similarly, 
$\theta_{y^{-1}}\Delta^{\mathfrak{p}}_w$
is isomorphic to $\theta_{y^{-1}}L_y$. However, we have
$\theta_{x^{-1}}L_x\cong P^\mathfrak{p}_d\cong \theta_{y^{-1}}L_y$,
where $d$ is the Duflo involution in the right cell of $x$,
see \cite[Section~3]{Ma2}. Note that
the latter cell coincides with the right cell of $y$.
This shows that Theorem~\ref{mainthm}\eqref{mainthm.2},
Theorem~\ref{mainthm}\eqref{mainthm.3} and 
Theorem~\ref{mainthm}\eqref{mainthm.4} fail.

Since Theorem~\ref{mainthm}\eqref{mainthm.2} fails, the birepresentations
$\cP/\mathrm{Ann}_{\cP}(\Delta^{\mathfrak{p}}_w)$
and $\mathrm{add}(\cP\, \Delta^{\mathfrak{p}}_w)$ cannot be
equivalent and hence Theorem~\ref{mainthm}\eqref{mainthm.1}
fails as well.

Finally, since $L_x$ and $L_y$ belong to different left cells inside
the same two-sided cell, their annihilators are incomparable primitive
ideals of minimal Gelfand-Kirillov dimension (among all other 
primitive ideals for $\mathcal{O}_0^\mathfrak{p}$). Therefore
the annihilator of $\Delta_w^\mathfrak{p}$ must be contained in the intersection of 
these two ideals and hence cannot be a primitive ideal.
This shows that Theorem~\ref{mainthm}\eqref{mainthm.6} fails
and completes Case~1.

\subsection{Case~2: $k+1=j$} This case follows from Case~1 using
the symmetry of the root system.

\subsection{Case~3: $i>k$}

In this case $\tau=s_{k+1}s_{k+2}\dots s_{j-1}s_{k}s_{k+1}\dots s_{i-1}$ 
and the Bruhat complement $[e,s_k\tau]\setminus[e,\tau]$ 
contains the following bigrassmannian elements:
$s_{k}s_{k+1}\dots s_{i}$, $s_ks_{k+1}\dots s_{i+1}$\dots, 
$s_ks_{k+1}\dots s_{j-1}$. Note that there are at least two such elements
(since $j\neq i+1$). Applying to them the arguments from
Case~1, we complete the proof.

\subsection{Case~4: $j<k+1$} This case follows from Case~3 using
the symmetry of the root system.

\subsection{Case~5: $i<k$ and $j>k+1$}

In this case $\tau=s_{k-1}s_{k-2}\dots s_{i}s_{k+1}s_{k+2}\dots s_{j-1}$ 
and the Bruhat complement $[e,s_k\tau]\setminus[e,\tau]$ 
contains the following bigrassmannian elements:
$s_{k}$, $s_{k}s_{k+1}$,\dots, $s_ks_{k+1}\dots s_{j-1}$,
$s_ks_{k-1}$,\dots $s_ks_{k-1}\dots s_{i}$. 
Clearly, there are at least two such elements.
Applying to them the arguments from
Case~1, we complete the proof.

\subsection{Asymptotic}\label{s5.5}

From Theorem~\ref{mainthm}, we see that the ratio of positive vs 
negative cases equals 
\begin{displaymath}
\frac{|\mathcal{X}_{n,k}^+|}{|\mathcal{X}_{n,k}^-|} =
\frac{n-1}{\binom{n}{2}-(n-1)}.
\end{displaymath}
This goes to $0$ when $n$ goes to infinity, which aligns
with the results of \cite[Section~6]{MMM}. It is also an interesting
observation that this  ratio does not depend on $k$.

\section{Main results: maximal parabolic subalgebras}\label{s9}

\subsection{Setup}\label{s9.1}

For $1\leq k<n$, set $m=n-k$. Let $\mathfrak{q}=\mathfrak{q}_k$ be the
unique parabolic subalgebra of $\mathfrak{sl}_n$ whose Levi factor is
$\mathfrak{sl}_k\oplus\mathfrak{sl}_m$, with $\mathfrak{sl}_k$ 
adjusted at the top left corner. This means that the only simple
reflection that is missed by $\mathfrak{q}_k$ is $(k,k+1)$.

\subsection{Weights and oriented cup diagrams}\label{s9.2}

We will now present a concise version of the combinatorial diagrammatic 
description of the category $\mathcal{O}_0^\mathfrak{q}$ from 
\cite{BS11a,BS11b}, so we refer the reader to these paper for 
all technical details.

We denote by $\mathcal{W}_{n,k}$ the set of all words of length $n$
in the alphabet with two letters, $\vee$ and $\wedge$, in which the
letter $\wedge$ appears exactly $k$ times. The letter $\vee$ should 
be thought of as the head of an arrow pointing down, while the letter
$\wedge$ should be though of as the head of an arrow pointing up.
For example, we have:
\begin{displaymath}
\mathcal{W}_{4,2}=\{\wedge\wedge\vee\vee,\,\,
\wedge\vee\wedge\vee,\,\,\wedge\vee\vee\wedge,\,\,
\vee\wedge\wedge\vee,\,\,\vee\wedge\vee\wedge,\,\,\vee\vee\wedge\wedge
\} 
\end{displaymath}
We call the unique word in $\mathcal{W}_{n,k}$ which starts with $k$ 
wedges {\em dominant}. Hence, in the above example, the dominant word
is $\wedge\wedge\vee\vee$. 

The group $\mathbf{S}_n$ acts {on} $\mathcal{W}_{n,k}$ by permuting the positions
of the letters in a word. This action induces a natural bijection $\Phi$
between $W^\mathfrak{q}_\mathrm{short}$ and $\mathcal{W}_{n,k}$ which sends
$w\in W^\mathfrak{q}_\mathrm{short}$ to the image of the dominant word
under $w^{-1}$.

Given an element $\lambda\in\mathcal{W}_{n,k}$, we can form an
{\em oriented cup diagram} by attaching to letters of $\lambda$ ends
of vertically falling strings and cups that altogether form a planar 
and oriented diagram. In particular, the requirement to be oriented
means that the two ends of a cup should be attached to different
letters (i.e., one to a $\wedge$ and the other one to a $\vee$).
For example, here are the six oriented cup diagrams that can be drawn
for the elements $\wedge\vee\wedge\vee$:
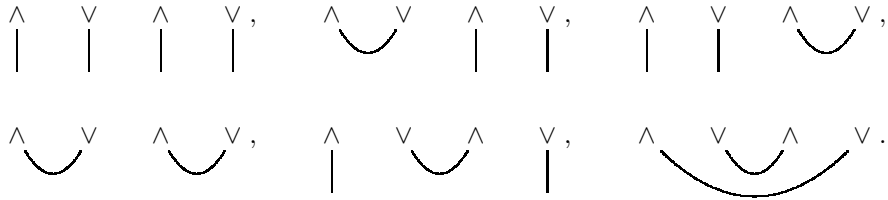
\begin{displaymath}
\xymatrix@R=5mm@C=5mm{\wedge\ar@{-}[d]&\vee\ar@{-}[d]&\wedge\ar@{-}[d]&\vee\ar@{-}[d]\\&&&},
\qquad
\xymatrix@R=5mm@C=5mm{\wedge\ar@/_5mm/@{-}[r]&\vee&\wedge\ar@{-}[d]&\vee\ar@{-}[d]\\&&&},
\qquad
\xymatrix@R=5mm@C=5mm{\wedge\ar@{-}[d]&\vee\ar@{-}[d]&\wedge\ar@/_5mm/@{-}[r]&\vee\\&&&},
\end{displaymath}
\begin{displaymath}
\xymatrix@R=5mm@C=5mm{\wedge\ar@/_5mm/@{-}[r]&\vee&\wedge\ar@/_5mm/@{-}[r]&\vee\\&&&},
\qquad
\xymatrix@R=5mm@C=5mm{\wedge\ar@{-}[d]&\vee\ar@/_5mm/@{-}[r]&\wedge&\vee\ar@{-}[d]\\&&&},
\qquad
\xymatrix@R=5mm@C=5mm{\wedge\ar@/_8mm/@{-}[rrr]&\vee\ar@/_5mm/@{-}[r]&\wedge&\vee\\&&&}.
\end{displaymath}

The {\em degree} of an oriented cup diagram is the number of cups oriented
clockwise. 

\subsection{Standard modules}\label{s9.3}

An oriented cup diagram is called {\em admissible} provided that this diagram does not
contain any vertical strand under $\vee$ to the left of any vertical strand
under $\wedge$. For example, for the element $\wedge\vee\wedge\vee$, out of the six
oriented cup diagram above, all but the first one (i.e. but the one with four 
vertical strands) are admissible.

For each $\lambda\in\mathcal{W}_{n,k}$, there is a unique admissible oriented
cup diagram $\mathbf{d}(\lambda)$ for $\lambda$ of degree zero. It can be constructed 
recursively  using the following algorithm. If $\lambda$ does not contain any 
$\vee$ to the left of some $\wedge$, then $\mathbf{d}(\lambda)$ consists of 
vertical strings. Otherwise, $\lambda$ must contain a subword $\vee\wedge$
which we connect by a cup which becomes oriented counter-clockwise.
We can now remove this cup and proceed recursively.

Conversely, given an unoriented cup diagram $\mathbf{d}$ with at most 
$\min(k,m)$ cups, there is a unique $\lambda\in\mathcal{W}_{n,k}$ such that,
removing the orientation from $\mathbf{d}(\lambda)$, we obtain $\mathbf{d}$.
In order to construct $\lambda$, we orient each cup in $\mathbf{d}$
counter-clockwise and then write the remaining $\wedge$'s on the left 
and the remaining $\vee$'s on the right of the remaining places for letters.
For examples, if $n=7$ and $k=3$, then, for the diagram
\begin{displaymath}
\mathbf{d}=\xymatrix@C=5mm@R=5mm{\ar@/_7mm/@{-}[rrr]&\ar@/_5mm/@{-}[r]
&&&\ar@{-}[d]&\ar@/_5mm/@{-}[r]&\\&&&&&&}, 
\end{displaymath}
we get the corresponding element $\vee\vee\wedge\wedge\vee\vee\wedge\in\mathcal{W}_{7,3}$.

The following proposition is a reformulation (with adaptation to
our terminology) of \cite[Theorem~5.1]{BS11a}.

\begin{proposition}\label{prop9.11}
Let $x,y\in W^\mathfrak{q}_\mathrm{short}$ and $i\in\mathbb{Z}_{\geq 0}$. 
Then the simple module $L_y$ appears as a simple subquotient of the 
parabolic Verma module $\Delta^\mathfrak{q}_x$ in degree $i$ 
(and then, necessarily, with multiplicity one) if and only if
there is an admissible oriented cup diagram of degree $0$ for $\Phi(x)$
such that the associated unoriented cup diagram coincides with the
unoriented cup diagram associated with $\mathbf{d}(\Phi(y))$,
while the latter has degree $i$.
\end{proposition}

Let us consider the example of $n=4$ and $k=2$. In this example, we have
\begin{displaymath}
W^\mathfrak{q}_\mathrm{short}=
\{e,s_2,s_2s_1,s_2s_3,s_2s_1s_3,s_2s_1s_3s_2\}.
\end{displaymath}
For the element $x=e$, here are the corresponding admissible oriented cup diagrams:
\begin{displaymath}
\xymatrix@R=5mm@C=5mm{\wedge\ar@{-}[d]&\wedge\ar@{-}[d]&\vee\ar@{-}[d]&\vee\ar@{-}[d]\\&&&},
\qquad
\xymatrix@R=5mm@C=5mm{\wedge\ar@{-}[d]&\wedge\ar@/_5mm/@{-}[r]&\vee&\vee\ar@{-}[d]\\&&&},
\qquad
\xymatrix@R=5mm@C=5mm{\wedge\ar@/_7mm/@{-}[rrr]&\wedge\ar@/_5mm/@{-}[r]&\vee&\vee\\&&&}.
\end{displaymath}
The corresponding degrees, from left to right are: $0$, $1$ and $2$.
For each underlying unoriented cup diagram, we can now write the corresponding words
for which the diagram becomes of degree zero:
\begin{displaymath}
\xymatrix@R=5mm@C=5mm{\wedge\ar@{-}[d]&\wedge\ar@{-}[d]&\vee\ar@{-}[d]&\vee\ar@{-}[d]\\&&&},
\qquad
\xymatrix@R=5mm@C=5mm{\wedge\ar@{-}[d]&\vee\ar@/_5mm/@{-}[r]&\wedge&\vee\ar@{-}[d]\\&&&},
\qquad
\xymatrix@R=5mm@C=5mm{\vee\ar@/_7mm/@{-}[rrr]&\vee\ar@/_5mm/@{-}[r]&\wedge&\wedge\\&&&}.
\end{displaymath}
We see that the first word correspond to $e$, the second to $s_2$ and the third
to $s_2s_1s_3s_2$. We conclude that $\Delta_e^\mathfrak{q}$ has
$L_e$ in degree $0$, then it has $L_{s_2}$ in degree $1$ and, finally,
$L_{s_2s_1s_3s_2}$ in degree $2$.

For the element $x=s_2$, we saw five admissible oriented 
cup diagrams in Subsection~\ref{s9.2}. From this
we conclude that $\Delta^\mathfrak{q}_{s_2}$ has $L_{s_2}$ in degree $0$,
then $L_{s_2s_1}$, $L_{s_2s_3}$ and $L_{s_2s_1s_3s_2}$ in degree $1$ and,
finally, $L_{s_2s_1s_3}$ in degree $2$.

\subsection{Thin standard modules}\label{s9.4}

Set $\mathtt{a}=\min(k,m)$. For $w\in W^\mathfrak{q}_\mathrm{short}$, we will say
that the module $\Delta^\mathfrak{q}_w$ is {\em thin} provided that there is a
unique $y\in W^\mathfrak{q}_\mathrm{short}$ such that
\begin{itemize}
\item $L_y$ is a subquotient of $\Delta^\mathfrak{q}_w$;
\item the diagram $\mathbf{d}(\Phi(y))$ contains exactly $\mathtt{a}$ cups.
\end{itemize}
We note that, for $w\in W^\mathfrak{q}_\mathrm{short}$, the number of cups in 
$\mathbf{d}(\Phi(w))$ coincides with the value of Lusztig's $\mathbf{a}$-function
on $w$, see \cite{Lu}. Hence the value $\mathtt{a}$ is the maximum value which 
the $\mathbf{a}$-function attains at the elements of $W^\mathfrak{q}_\mathrm{short}$.
By a result of Irving, see \cite[Proposition~4.3]{Ir},  any socular constituent
$L_y$ of any $\Delta^\mathfrak{q}_w$, where $w\in W^\mathfrak{q}_\mathrm{short}$,
has the property that the diagram $\mathbf{d}(\Phi(y))$ contains exactly $\mathtt{a}$ cups.
Consequently, a thin parabolic Verma module must have simple socle.

It is easy to check, by a direct computation, that in our running example
$n=4$ and $k=2$, the list of all thin parabolic Verma modules looks as follows:
\begin{displaymath}
\Delta^\mathfrak{q}_e,\quad
\Delta^\mathfrak{q}_{s_2s_1},\quad
\Delta^\mathfrak{q}_{s_2s_3},\quad
\Delta^\mathfrak{q}_{s_2s_1s_3s_2}.
\end{displaymath}

\subsection{Formulation}\label{s9.5}

Let $n$ and $k$ be as above. Denote by 
$\mathcal{Y}_{n,k}$ the set of all elements $w\in W^\mathfrak{q}_\mathrm{short}$
such that the word $\Phi(w)$ has the following property:
\begin{itemize}
\item if $k< n-k$, then all $\vee$'s in $w$ appear next to each other;
\item if $k> n-k$, then all $\wedge$'s in $w$ appear next to each other;
\item if $k=n-k$, then either all $\vee$'s or all $\wedge$'s in $w$
appear next to each other.
\end{itemize}
For example, if $n=5$ and $k=2$, then $k=2<3=n-k$ and
\begin{displaymath}
\mathcal{Y}_{5,2}=\{\wedge\wedge\vee\vee\vee,\quad\wedge\vee\vee\vee\wedge,\quad
\vee\vee\vee\wedge\wedge\}. 
\end{displaymath}
At the same, time, if $n=4$ and $k=2$, then $k=n-k$ and we have:
\begin{displaymath}
\mathcal{Y}_{4,2}=\{\wedge\wedge\vee\vee,\quad\wedge\vee\vee\wedge,\quad
\vee\wedge\wedge\vee,\quad\vee\vee\wedge\wedge\}. 
\end{displaymath}

We can now formulate our main result for maximal parabolic subalgebras.

\begin{theorem}\label{mainthm2}
Let $n$ and $k$ be as above.

\begin{enumerate}[$($a$)$]
\item\label{maintheorem2.1} If $k\in\{1,n-1,\frac{n}{2}\}$, then the only 
$w\in W^\mathfrak{q}_\mathrm{short}$ for which $\mathbf{K}(\Delta^\mathfrak{q}_w)=
\mathtt{true}$ are $w=e$ and $w=w_0^\mathfrak{q}w_0$.
\item\label{maintheorem2.2} If $k\neq\frac{n}{2}$, 
then, for $w\in W^\mathfrak{q}_\mathrm{short}$, we have
$\mathbf{K}(\Delta^\mathfrak{q}_w)=\mathtt{true}$ if and only if $w\in\mathcal{Y}_{n,k}$.
\end{enumerate}
\end{theorem}

We note that the case $k\in\{1,n-1\}$ is covered by both statements 
of Theorem~\ref{mainthm2}.

The next two subsections are dedicated to prove Theorem~\ref{mainthm2}.

\subsection{Non-thin parabolic Verma modules are Kostant negative}\label{s9.6}

If $\Delta_w^\mathfrak{q}$ is not thin, for some $w\in W^\mathfrak{q}_\mathrm{short}$,
then there exist different $x,y\in W^\mathfrak{q}_\mathrm{short}$ 
of maximal $\mathbf{a}$-value such that both
$L_x$ and $L_y$ are subquotients of $\Delta_w^\mathfrak{q}$ (necessarily with
multiplicity $1$). Let $d$ be the unique involution in the Kazhdan-Lusztig
right cell of $x$ (which coincides with the corresponding cell for $y$).
Then
\begin{displaymath}
\theta_{x^{-1}}\Delta_w^\mathfrak{q}\cong
\theta_{x^{-1}}L_x\cong
\theta_dL_d\cong \theta_{y^{-1}}L_y\cong
\theta_{y^{-1}}\Delta_w^\mathfrak{q}.
\end{displaymath}
Since $\theta_{x^{-1}}\not\cong\theta_{y{-1}}$ while
$\theta_{x^{-1}}\Delta_w^\mathfrak{q}\cong \theta_{y^{-1}}\Delta_w^\mathfrak{q}$,
from Subsection~\ref{s3.5} it follows that 
$\mathbf{K}(\Delta_w^\mathfrak{q})=\mathtt{false}$.

\subsection{Kostant problem for thin parabolic Verma modules}\label{s9.7}

After the observation in the previous subsection, we are left to consider
thin parabolic Verma modules. We start with a classification of such modules.

\begin{proposition}\label{prop9.17}
For $w\in W^\mathfrak{q}_\mathrm{short}$, the module
$\Delta_w^\mathfrak{q}$ is thin if and only if $w\in \mathcal{Y}_{n,k}$.
\end{proposition}

\begin{proof}
For $\lambda\in\mathcal{W}_{n,k}$, define the {\em signature} of $\lambda$
as the vector $(a_1,a_2,\dots,a_l)$ of positive integers where
\begin{itemize}
\item $\lambda$ starts on the left with $a_1$ equal letter
$\vee$ (or $\wedge$);
\item $\lambda$ continues with $a_2$ equal letters $\wedge$ (resp. $\vee$);
\item and so on.
\end{itemize}
For example, the signature of $\vee\wedge\wedge\vee\wedge\wedge\wedge$
equals $(1,2,1,3)$. The number $l$ will be called the {\em flip number},
which is $4$ in our example.

It is straightforward to check that all $\lambda$ having the flip number
at most two index thin parabolic Verma modules. So, let us take 
some $\lambda$ with flip number $l>2$. Assume that there exists
$1<i<l$ such that $a_i<a_{i+1}+a_{i-1}$. We claim that in this case 
$\lambda$ indexes a parabolic Verma module that is not thin.

To prove this we need to construct at least two different oriented
cup diagrams with $\min(k,n-k)$ cups for this $\lambda$. Take
any connected subword $\wedge\vee$ or $\vee\wedge$ in which one of the
letters is in the $a_i$-part of $\lambda$, connect the two letters in this
subword by a cup and remove it. In this way we reduce $a_i$ by $1$
and we also reduce either $a_{i+1}$ or $a_{i-1}$ by $1$. We can do
this recursively such that we reduce $a_i$ to $1$ while keeping the
new $a_{i\pm 1}$'s positive. This gives a connected subword of the
form either $\vee\wedge\vee$ or $\wedge\vee\wedge$. Here we clearly 
see that we can connect the middle letter by a cup either to the 
letter on the right or to the letter on the left. This gives
rise to two different oriented cup diagrams as desired.

So, if $\lambda$ indexes a thin parabolic Verma module, then
$a_i\geq a_{i+1}+a_{i-1}$, for all $1<i<l$. In particular,
$a_i\geq a_{i+1}$ and $a_i\geq a_{i-1}$. Hence $l\leq 3$.
As the case $l\leq 2$ is already settled, it remains to 
consider the case $l=3$.

In the case $a_2< a_1+a_3$, the above argument implies
that $\lambda$ indexes a parabolic Verma module that is not think.
If $a_2\geq a_1+a_3$, we get $\lambda\in\mathcal{Y}_{n,k}$
by definition. It is easy to see that such $\lambda$ do indeed
index thin parabolic Verma modules.
\end{proof}

Now we need to consider two cases. Recall that, to avoid trivial cases, 
we assume $1\leq k\leq n-1$. Let $w$ be such that $\lambda=\Phi(w)$. Note that if the flip number of $\lambda$ is equal to $2$ then either $w=e$ or $w=w_0^{ \mathfrak{q}}w_0$. In the first case the corresponding parabolic Verma module is a quotient of projective Verma module, and hence it is Kostant positive. In the second  case the corresponding parabolic Verma module is simple which is Kostant positive. Thus we assume that the flip number of $\lambda$ is equal to $3$.

{\bf Case~1.} Assume $k=\frac{n}{2}$. In this case we have $a_2=a_1+a_3$.
This implies that any admissible oriented cup diagram for
$\lambda$ with $k$ cups must have two non-nested cups next
to each other. By the main result
of \cite{MMM}, the corresponding simple module $L_u$, which is the socle
of our parabolic Verma module, is Kostant negative.

Let us now analyze the module $\Delta^\mathfrak{q}_w$ in more detail.
There is a unique admissible oriented cup diagram for $\lambda$ with $k$
cups. Any other admissible oriented cup diagrams for $\lambda$ is 
obtained from this one by replacing some outer clock-wise oriented cups
by vertical strings. Here is an example, for $n=8$ and $k=4$,
of an original cup diagram
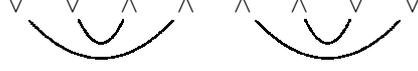
\begin{displaymath}
\xymatrix@C=3mm@R=3mm{
\vee\ar@/_7mm/@{-}[rrr]&\vee\ar@/_5mm/@{-}[r]&\wedge&\wedge
&\wedge\ar@/_7mm/@{-}[rrr]&\wedge\ar@/_5mm/@{-}[r]&\vee&\vee\\
&&&&&&&
}
\end{displaymath}
and the cup diagrams that can be obtained
from it by replacing some outer clock-wise oriented cups:
\begin{displaymath}
\xymatrix@C=3mm@R=3mm{
\vee\ar@/_7mm/@{-}[rrr]&\vee\ar@/_5mm/@{-}[r]&\wedge&\wedge
&\wedge\ar@{-}[d]&\wedge\ar@/_5mm/@{-}[r]&\vee&\vee\ar@{-}[d]\\
&&&&&&&
}\quad\text{ and }\quad\xymatrix@C=3mm@R=3mm{
\vee\ar@/_7mm/@{-}[rrr]&\vee\ar@/_5mm/@{-}[r]&\wedge&\wedge
&\wedge\ar@{-}[d]&\wedge\ar@{-}[d]&\vee\ar@{-}[d]&\vee\ar@{-}[d]\\
&&&&&&&
}
\end{displaymath}

In \cite[Section~5.5]{MMM} one can find an explicit 
construction of two different elements $x$ and $y$ in $W$
such that $\theta_x L_u\cong \theta_y L_u$. In the case of the
above (generic) example, these elements $x$ and $y$ are given 
by the following diagrams:
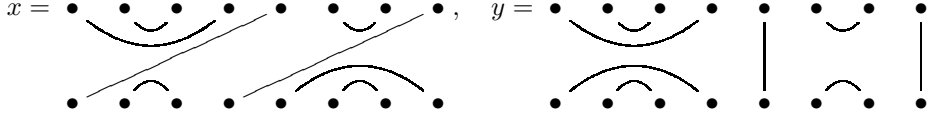
\begin{displaymath}
x=\xymatrix@C=3mm@R=9mm{\bullet\ar@/_5mm/@{-}[rrr]&\bullet\ar@/_3mm/@{-}[r]
&\bullet&\bullet&\bullet
&\bullet\ar@/_3mm/@{-}[r]&\bullet&\bullet\\
\bullet\ar@{-}[rrrru]&\bullet\ar@/^3mm/@{-}[r]&\bullet
&\bullet\ar@{-}[rrrru]&\bullet\ar@/^5mm/@{-}[rrr]&
\bullet\ar@/^3mm/@{-}[r]&\bullet&\bullet},\quad
y=\xymatrix@C=3mm@R=9mm{\bullet\ar@/_5mm/@{-}[rrr]&\bullet\ar@/_3mm/@{-}[r]
&\bullet&\bullet&\bullet\ar@{-}[d]
&\bullet\ar@/_3mm/@{-}[r]&\bullet&\bullet\ar@{-}[d]\\
\bullet\ar@/^5mm/@{-}[rrr]&\bullet\ar@/^3mm/@{-}[r]&\bullet
&\bullet&\bullet&\bullet\ar@/^3mm/@{-}[r]&\bullet&\bullet}
\end{displaymath}
Note the  outer right lower cup in $x$. For all admissible oriented cup diagrams
for $\lambda$, except the original one, this cup hits vertical strands.
This means that $\theta_x$ kills the corresponding simple module, in other words,
$\theta_x L_u\cong \theta_x\Delta_w^\mathfrak{q}$.
Similar arguments apply to general $w$ leading to the same conclusion
$\theta_x L_u\cong \theta_x\Delta_w^\mathfrak{q}$.

Since $\theta_x L_u\cong \theta_y L_u$
and $\theta_y L_u$ is a submodule of $\theta_y\Delta_w^\mathfrak{q}$, we thus 
have a non-zero degree zero morphism from $\theta_x\Delta_w^\mathfrak{q}$
to $\theta_y\Delta_w^\mathfrak{q}$. As $\theta_x$ and $\theta_y$ are
indecomposable and non-isomorphic and the endomorphism algebra of projective
functors is positively graded, it follows that this degree zero zero map 
from $\theta_x\Delta_w^\mathfrak{q}$
to $\theta_y\Delta_w^\mathfrak{q}$ cannot be the evaluation of some 
map from $\theta_x$ to $\theta_y$ at $\Delta_w^\mathfrak{q}$.
From Subsection~\ref{s3.5} we therefore conclude that the answer to
Kostant's problem for $\Delta_w^\mathfrak{q}$ is negative.

{\bf Case~2.} Assume $k\neq \frac{n}{2}$. In this case we have
$a_2>a_1+a_3$. This implies that any oriented cup diagram for
$\lambda$ with $\min(k,n-k)$ cups must have a vertical strand
that separates the two sets of nested cups. By the main result
of \cite{MMM}, the corresponding simple module $L_u$, which is the socle
of our parabolic Verma module, is Kostant positive. For different
$x$ and $y$ in $W$, we have an injective restriction map from
$\mathrm{Hom}_\mathfrak{g}(\theta_x\Delta_w^\mathfrak{q},\theta_y\Delta_w^\mathfrak{q})$
to $\mathrm{Hom}_\mathfrak{g}(\theta_xL_u,\theta_yL_u)$. 

Since $L_u$ is Kostant positive, $\mathrm{Hom}_{\ccP}(\theta_x,\theta_y)$
surjects on the latter and hence also on the former. This implies
that $\Delta_w^\mathfrak{q}$ is Kostant positive.

This completes the proof.

\subsection{Example: the extreme maximal parabolic}\label{s9.8}

Consider the parabolic subalgebra $\mathfrak{p}$ of $\mathfrak{sl}_n$ 
corresponding to the choice of all simple roots but the first one. Then
the Levi factor of $\mathfrak{p}$ is isomorphic to $\mathfrak{sl}_{n-1}$.
Setting $s_i=(i,i+1)$, we have:
\begin{displaymath}
W^{\mathfrak{p}}_{\mathrm{short}}=
\{e,s_1,s_1s_2,\dots,s_1s_2s_{3}\cdots s_{n-1}\}.
\end{displaymath}
It is well-known, see, for example, \cite{St2}, that the corresponding 
category $\mathcal{O}_0^\mathfrak{p}$ is equivalent to the category of
modules over the following quiver
\begin{displaymath}
\xymatrix{
n-1\ar@/^3pt/[r]^{\alpha_{n-1}}&n-2\ar@/^3pt/[r]^{\alpha_{n-2}}\ar@/^3pt/[l]^{\beta_{n-1}}&
\dots\ar@/^3pt/[r]^{\alpha_3}\ar@/^3pt/[l]^{\beta_{n-2}}
&2\ar@/^3pt/[r]^{\alpha_2}\ar@/^3pt/[l]^{\beta_3}&
1\ar@/^3pt/[r]^{\alpha_1}\ar@/^3pt/[l]^{\beta_2}&0\ar@/^3pt/[l]^{\beta_1}
}
\end{displaymath}
with the following relations:
\begin{displaymath}
\alpha_1\beta_1=0,\quad
\alpha_i\alpha_{i+1}=0,\quad
\beta_i\beta_{i-1}=0,\quad
\alpha_i\beta_i=\beta_{i-1}\alpha_{i-1}.
\end{displaymath}
Here the vertex $0$ corresponds to $e$ and, for $i>0$, the vertex $i$
corresponds to $s_1s_2\cdots s_i$.

In this relation, the parabolic Verma modules are the standard modules with the
(unique) quasi-hereditary structure:
\begin{displaymath}
\Delta_0=\xymatrix{0\ar[d]\\1},\qquad 
\Delta_1=\xymatrix{1\ar[d]\\2},\qquad 
\Delta_{n-2}=\xymatrix{n-2\ar[d]\\n-1},\qquad 
\Delta_{n-1}=\xymatrix{n-1}. 
\end{displaymath}

The $\mathbf{a}$-value of $e$ is zero, while the
$\mathbf{a}$-value of all other elements in 
$(\mathbf{S}_n)^\mathfrak{p}_\mathrm{short}$ is $1$.
It follows that the only thin parabolic Verma modules are
$\Delta_0$ and $\Delta_{n-1}$. From Theorem~\ref{mainthm2}
we obtain that these two parabolic Verma modules are the only
Kostant positive parabolic Verma modules in this case.

In fact, in this case one can extend Theorem~\ref{mainthm2}
in the following way:

\begin{proposition}\label{prop9.1}
For $w\in W^{\mathfrak{p}}_{\mathrm{short}}$, the following assertions are equivalent:
\begin{enumerate}[$($a$)$]
\item\label{prop9.1.1} $\mathbf{K}(\Delta^{\mathfrak{p}}_w)=\mathtt{true}$.
\item\label{prop9.1.2} For $x,y\in W$ such that $x\neq y$, 
$\theta_x \Delta^{\mathfrak{p}}_w\neq 0$ and
$\theta_y \Delta^{\mathfrak{p}}_w\neq 0$, we have $\theta_x \Delta^{\mathfrak{p}}_w\not\cong \theta_y \Delta^{\mathfrak{p}}_w$ (as ungraded modules).
\item\label{prop9.1.3} For all $x,y\in W$ such that $x\neq y$, 
$\theta_x \Delta^{\mathfrak{p}}_w\neq 0$ and
$\theta_y \Delta^{\mathfrak{p}}_w\neq 0$, we have 
$[\theta_x \Delta^{\mathfrak{p}}_w]\neq [\theta_y \Delta^{\mathfrak{p}}_w]\langle i\rangle$,
for $i\in\mathbb{Z}$,
in $\mathbf{Gr}[\mathcal{O}_0^\mathbb{Z}]$.
\item\label{prop9.1.4} For all $x,y\in W$ such that $x\neq y$, 
$\theta_x \Delta^{\mathfrak{p}}_w\neq 0$ and
$\theta_y \Delta^{\mathfrak{p}}_w\neq 0$, we have 
$[\theta_x \Delta^{\mathfrak{p}}_w]\neq [\theta_y \Delta^{\mathfrak{p}}_w]$
in $\mathbf{Gr}[\mathcal{O}_0]$.
\item\label{prop9.1.5} $w\in\{e,s_1s_2\cdots s_{n-1}\}$.
\item\label{prop9.1.6} The annihilator of $\Delta^{\mathfrak{p}}_w$ in $U(\mathfrak{g})$
is a primitive ideal.
\end{enumerate}
\end{proposition}

\begin{proof}
We start with $w=e$. Then 
$\mathbf{K}(\Delta^{\mathfrak{p}}_e)=\mathtt{true}$ by Theorem~\ref{mainthm2}, which gives  Claim~\eqref{prop9.1.1}.
Applying (pairwise non-isomorphic) indecomposable projectives functors 
to $\Delta^{\mathfrak{p}}_e$ one gets either zero or (pairwise non-isomorphic)
indecomposable projectives in $\mathcal{O}_0^\mathfrak{p}$.
As the latter is a highest weight category, the images of these 
indecomposable projectives in the (graded) Grothendieck group are
linearly independent. This implies Claims~\eqref{prop9.1.2},
\eqref{prop9.1.3} and \eqref{prop9.1.4}. Since $L_e$ is the trivial module, the
annihilator of $\Delta^{\mathfrak{p}}_e$ coincides with the annihilator of 
$L_{s_1}$ and hence is a primitive ideal, giving Claim~\eqref{prop9.1.6}.

Next consider $w=s_1s_2\cdots s_{n-1}$. 
Again, $\mathbf{K}(\Delta^{\mathfrak{p}}_{w_0})=\mathtt{true}$
by Theorem~\ref{mainthm2},
which gives  Claim~\eqref{prop9.1.1}.
Applying (pairwise non-isomorphic) indecomposable projectives functors 
to $L_{w_0^\mathfrak{p}w_0}$ one gets either zero or (pairwise non-isomorphic)
indecomposable tilting modules in $\mathcal{O}_0^\mathfrak{p}$.
As the latter is a highest weight category, the images of these 
indecomposable projectives in the (graded) Grothendieck group are
linearly independent. This implies Claims~\eqref{prop9.1.2},
\eqref{prop9.1.3} and \eqref{prop9.1.4}. Since $\Delta^{\mathfrak{p}}_w$ is 
simple, its annihilator is a primitive ideal, giving Claim~\eqref{prop9.1.6}.

For $w\not\in\{e,s_1s_2\cdots s_{n-1}\}$, the corresponding 
$\Delta^{\mathfrak{p}}_w$ is not thin. It
has top $L_{w}$ and socle $L_{ws_i}$, for some $i$.
We have $\theta_{s_i}L_w=0$ and $\theta_{s_{i-1}}L_{ws_i}=0$ while
$\theta_{s_i}L_{ws_i}\cong \theta_{s_{i-1}s_i}L_w$. This implies that 
Claims~\eqref{prop9.1.1}--\eqref{prop9.1.4} fail. Also, as the annihilators
of $L_{w}$ and $L_{ws_i}$ are incomparable primitive ideals, it follows
that the annihilator of $\Delta^{\mathfrak{p}}_w$ is contained in the
intersection of these two primitive ideals and hence is not a primitive ideal.
Therefore Claim~\eqref{prop9.1.6} fails as well.

This completes the proof. 
\end{proof}

\section{Upshot: general observations and speculations}\label{s8}

\subsection{Preliminary definitions}\label{s8.1}

For a composition $\mu\models n$, we denote by $\mathbf{c}(\mu)\vdash n$
the corresponding partition of $n$. The partition $\mathbf{c}(\mu)$
is obtained from $\mu$ by ordering the parts of the latter in a weakly
decreasing order. For two compositions $\mu,\nu\models n$, we write
$\mu\sim\nu$ provided that $\mathbf{c}(\mu)=\mathbf{c}(\nu)$.

There is a natural bijection between compositions of $n$ and parabolic
subalgebras of $\mathfrak{sl}_n$ given as follows, for $\mu=(\mu_1,\mu_2,\dots)$,
the corresponding subalgebra $\mathfrak{p}_\mu$ has diagonal blocks of size
$\mu_1,\mu_2,\dots$ reading along the main diagonal from north-west to
south-east.

For $\mu=(\mu_1,\mu_2,\dots,\mu_k)\models n$, define the sets
\begin{displaymath}
X_1=\{1,2,\dots,\mu_1\}, X_2=\{\mu_1+1,\mu_1+2,\dots,\mu_1+\mu_2\}
\quad\text{ and so on}.
\end{displaymath}
We will call these sets the {\em blocks of $\mu$}.
Denote by $G_\mu$  the subgroup of $\mathbf{S}_n$ consisting of all
permutations $\pi$ satisfying the following property: 
for each $i\in\{1,2,\dots,k\}$, there is some $j\in\{1,2,\dots,k\}$
such that $|X_i|=|X_j|$ and $\pi$ sends the elements of $X_i$ to the
elements of $X_j$ preserving the natural order between these elements.
Here is an example of an element in $G_{(2,2,1,2)}$, where 
$X_1=\{1,2\}$, $X_2=\{3,4\}$, $X_3=\{5\}$ and $X_4=\{6,7\}$:
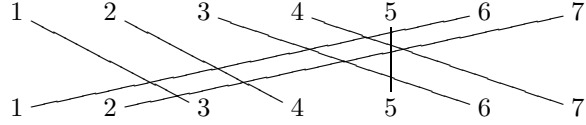
\begin{displaymath}
\xymatrix{
1\ar@{-}[drr]&2\ar@{-}[drr]&3\ar@{-}[drrr]&
4\ar@{-}[drrr]&5\ar@{-}[d]&6\ar@{-}[dlllll]&7\ar@{-}[dlllll]\\
1&2&3&4&5&6&7
}
\end{displaymath}

Let $\mu,\nu\models n$ be two compositions such that $\mu\sim\nu$.
Let $X_1,X_2,\dots,X_K$ be the blocks of $\mu$
and $Y_1,Y_2,\dots,Y_K$ be the blocks of $\nu$.
There is a unique $\tau\in S_k$ such that 
\begin{itemize}
\item $|X_i|=|Y_{\tau(i)}|$, for all $i$;
\item if $i<j$ and $|X_i|=|X_j|$, then $\tau(i)<\tau(j)$.
\end{itemize}
We denote by $\omega_{\mu,\nu}$ the unique element of $\mathbf{S}_n$
which sends, for each $i$, the elements of $X_i$ to the elements of 
$Y_{\tau(i)}$ preserving the natural order among these elements.
For example:
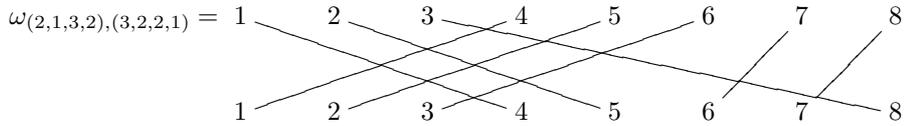
\begin{displaymath}
\omega_{(2,1,3,2),(3,2,2,1)}=
\xymatrix{
1\ar@{-}[drrr]&2\ar@{-}[drrr]&3\ar@{-}[drrrrr]&
4\ar@{-}[dlll]&5\ar@{-}[dlll]&6\ar@{-}[dlll]&7\ar@{-}[dl]&8\ar@{-}[dl]&\\
1&2&3&4&5&6&7&8
}
\end{displaymath}
Note that $\omega_{\mu,\nu}G_\mu=G_\nu\omega_{\mu,\nu}$.

The parabolic subgroup $W_\mu$ of $\mathbf{S}_n$ corresponding to the composition $\mu$
is the product $S_{X_1}\times S_{X_2}\times\dots\times S_{X_k}$ of symmetric groups.

\subsection{Some positive cases}\label{s8.2}

For $\mu\models n$, consider $\mathfrak{p}=\mathfrak{p}_\mu$.

\begin{proposition}\label{prop8.1}
For $w\in (\mathbf{S}_n)_\mathrm{short}^\mathfrak{p}$, we have
$\mathbf{K}(\Delta^\mathfrak{p}_w)=\mathtt{true}$ if
there is $\nu\models n$ such that $\mu\sim\nu$ and
$w\in G_\mu\omega_{\nu,\mu}$.
\end{proposition}

\begin{proof}
The proof is similar to the one in Subsection~\ref{s5.3}.
\end{proof}

Note that, if $\nu$ and $\nu'$ are different compositions of $n$ such that
$\mu\sim\mu$ and $\nu'\sim\mu$, then the sets
$G_\mu\omega_{\nu,\mu}$ and $G_\mu\omega_{\nu',\mu}$
are disjoint.

\subsection{Some negative cases}\label{s8.3}

For $\mu\models n$, consider $\mathfrak{p}=\mathfrak{p}_\mu$.
Let $w_0^\mu$ be the longest element in the parabolic subgroup
of $W$ corresponding to $\mathfrak{p}$. Let $\mathcal{R}$ denote
the right Kazhdan-Lusztig cell of $W$ containing the
element $w_0^\mu w_0$.

For $w\in (\mathbf{S}_n)_\mathrm{short}^\mathfrak{p}$, we say that the
parabolic Verma module $\Delta_w^\mathfrak{p}$ is {\em thin}
provided that there is a unique $u\in \mathcal{R}$ such that
$[\Delta_w^\mathfrak{p}:L_u]\neq 0$, moreover 
$[\Delta_w^\mathfrak{p}:L_u]=1$.

The following proposition is a general off-shot from
the arguments applied in Section~\ref{s9}.

\begin{proposition}\label{prop8.2}
For $w\in (\mathbf{S}_n)_\mathrm{short}^\mathfrak{p}$, we have
$\mathbf{K}(\Delta^\mathfrak{p}_w)=\mathtt{false}$
provided that $\Delta_w^\mathfrak{p}$ is not thin.
\end{proposition}

\begin{proof}
For $u\in\mathcal{R}$, we have $\theta_{u^{-1}}\Delta_w^\mathfrak{p}$
is isomorphic to a direct sum of $[\Delta_w^\mathfrak{p}:L_u]$
copies of $\theta_d L_d$, where $d$ is the Duflo involution in 
$\mathcal{R}$. From Subsection~\ref{s3.5} it follows that 
Kostant positivity of $\Delta_w^\mathfrak{p}$ implies that 
$\theta_{u^{-1}}\Delta_w^\mathfrak{p}$ is either indecomposable
or zero. In particular, we must have 
$[\Delta_w^\mathfrak{p}:L_u]\leq 1$, for each such $u$.

If we have two different $u,\hat{u}\in\mathcal{R}$ such that 
$[\Delta_w^\mathfrak{p}:L_u]=[\Delta_w^\mathfrak{p}:L_{\hat{u}}]=1$,
the we have $\theta_{u^{-1}}\Delta_w^\mathfrak{p}\cong
\theta_{\hat{u}^{-1}}\Delta_w^\mathfrak{p}$ which implies
that $\Delta_w^\mathfrak{p}$ is Kostant negative, again
by Subsection~\ref{s3.5}. This completes the proof.
\end{proof}

\subsection{Some speculations}\label{s8.4}

The above results reduce the study of Kostant problem for parabolic 
Verma module to the case of thin parabolic Verma modules that
are not covered by Subsection~\ref{s8.2}. Each thin
parabolic Verma module has simple socle. The results of
Section~\ref{s9} suggest that, for a thin parabolic Verma module
$\Delta_w^\mathfrak{p}$ with socle $L_u$, we have
$\mathbf{K}(\Delta_w^\mathfrak{p})=\mathbf{K}(L_u)$.
We do not know whether this is true in general and if yes,
how to prove this. In any case this is only a reduction to
the still open problem of determining $\mathbf{K}(L_u)$.

The above results also suggest that classification of thin
parabolic Verma modules is an interesting and important
problem.


\noindent
V.~M.: Department of Mathematics, Uppsala University, Box. 480,
SE-75106, Uppsala, SWEDEN, email: {\tt mazor\symbol{64}math.uu.se}

\noindent
S.~S.: Department of Mathematics, Indian Institute of Science, Bangalore, 560012,
 email: {\tt maths.shraddha\symbol{64}gmail.com}

\end{document}